\documentclass[reqno]{amsart}
\usepackage{mathrsfs}
\newtheorem{thm}{Theorem}[section]
\newtheorem{cor}[thm]{Corollary}
\newtheorem{lem}[thm]{Lemma}
\newtheorem{prop}[thm]{Proposition}
\newtheorem{other}{\bf Theorem}

\newtheorem{otherl}{\bf Lemma}

\theoremstyle{definition}

\theoremstyle{remark}
\newtheorem{rem}[thm]{Remark}
\numberwithin{equation}{section}

\newcommand{\eps}{\varepsilon}

\allowdisplaybreaks

\newcommand{\bn}{{\mathbb B_N}}
\newcommand{\sn}{{\mathbb S_N}}
\newcommand{\cn}{\mathbb C^N}
\newcommand\dsi{\,d\sigma}

\usepackage{hyperref}
\begin{document}

\title[ Besov-Lipschitz and mean Besov-Lipschitz spaces ]
{ Besov-Lipschitz and mean Besov-Lipschitz spaces of holomorphic functions on the unit ball}%


\author{Miroljub Jevti\' c}
\address{Matemati\v cki Fakultet, PP. 550, 11000 Belgrade, Serbia}
\email{jevtic@matf.bg.ac.yu}

\thanks{The authors are supported by MN Serbia, Project  ON144010}%


\author{Miroslav Pavlovi\'c}%
\address{Faculty of Mathematics, Studentski trg 16, 11001 Belgrade, p.p. 550, Serbia}%
\email{pavlovic@matf.bg.ac.rs}%


\begin{abstract}
We give several characterizations of holomorphic mean
Besov-Lipschitz space on the unit ball in $\cn $ and appropriate
Besov-Lipschitz space and prove the equivalences between them.
Equivalent norms on the mean Besov-Lipschitz space involve
different types of $L^p$-moduli of continuity, while in
characterizations of Besov-Lipschitz space we use not only the
radial derivative but also the gradient and the tangential
derivatives. The characterization in terms of the best
approximation by polynomials is also given.
\end{abstract}

\maketitle
\section{Introduction and main results}
Let $\bn$ denote the unit ball in $\cn$ and let $\sn=\partial
\bn$, where $N$ is a positive integer. For a point
$z=(z_1,\ldots,z_N)\in \cn$ we write
$|z|=(|z_1|^2+\ldots+|z_N|^2)^{1/2}$. The normalized Lebesgue
measures on $\bn$ and $\sn$ will be denoted by $dv$ and $d\sigma$,
respectively.
 The $L^p$-mean over the sphere $|z|=r$ $(0<r<1)$ of a  Borel function $f$ on $\bn$
    is defined by
\[\begin{aligned} M_p(r,f)= \left(\int_{\sn }|f(r\zeta)|^p\dsi(\zeta) \right)^{1/p}
 \quad (1\le p<\infty)\end{aligned} \]
and $M_\infty(r,f)=\mathop{\rm ess\,sup}_{\zeta_\in\sn}|f(r\zeta)|.$
The class $L^{p,q}_\alpha(\bn)$, $\alpha\in \mathbb R,$ $1\le q\le\infty,$ consists of
 those functions $f$ for which
\begin{equation}\label{}\begin{aligned}
\|f\|_{p,q,\alpha}:=\left(2N\int_0^1
M_p^q(r,f)(1-r^2)^{q\alpha-1}r^{2N-1}\,dr\right)^{1/q}\ <\infty.
 \end{aligned} \end{equation}
 In the case $q=\infty$ this should be interpreted as
 \[\|f\|_{p,\infty,\alpha}:=\mathop{\rm ess\,sup}_{0\le r<1}(1-r^2)^\alpha M_p(r,f)\ <\infty.\]

Let $H(\bn)$ denote the space of all holomorphic functions in
$\bn$. For $1\le p\le\infty$ the Hardy space
 $H^p(\bn)=H^p$ consists of all functions $f\in H(\bn)$ such that
\[\begin{aligned} \|f\|_p:=\sup_{0<r<1}M_p(r,f)\ <\infty,  \end{aligned} \]
i.e. $H^p(\bn)=L^{p,\infty}_0(\bn)\cap H(\bn).$ For information on
Hardy spaces of several variables we refer to   \cite{20}.

We are interested in the mixed-norm Bergman space
\[\begin{aligned}
 H^{p,q,\alpha }=H^{p,q,\alpha }(\bn)=L^{p,q}_\alpha\cap H(\bn)
  \end{aligned} \]
and some closely related spaces. Note that  $H^{p,q,\alpha
}=\{0\}$ for $\alpha\le 0$, if $1\leq q<\infty $, and for $\alpha
<0$ if $q=\infty .$
 By using standard arguments one proves that $H^{p,q, \alpha }$ is complete.
  These spaces arise naturally in the study of Hardy spaces
(see  \cite{7}). Note also that $(1.1)$, for $1\leq p=q<\infty $,
reduces to
\[\begin{aligned}
\|f\|_{p,p,\alpha}=\int_\bn |f(z)|^p(1-|z|^2)^{p\alpha-1}\,dv(z).
   \end{aligned} \]
Therefore, $H^{p,p,\alpha }$ is a weighted Bergman space.

\emph{Convention}. From now on, unless specified otherwise, we
will assume $1\le p,q\le\infty,$  $\alpha>0,$ and $\nu$, $N$ are
positive integers.

\subsection{Besov-Lipschitz spaces}
In order to give one of several possible definitions of
Besov-Lipschitz spaces we need some more notation. The radial
derivative  $\mathcal Rf,$ where  $f\in H(\bn)$, is defined as
\[\mathcal Rf(z)=\sum_{j=1}^N z_j \frac{\partial f}{\partial z_j}(z),\footnotemark{} \]
\footnotetext{In \cite{6}, $\mathcal Rf$ is defined as
$f(0)+\sum_{j=1}^N z_j\partial f/\partial z_j$.} which can also be
written as
\[\mathcal Rf(z)=\sum_{k=1}^\infty k f_k(z),\]
where
\(f(z)=\sum_{k=0}^\infty  f_k(z)\)
is the homogeneous expansion of $f.$
Using the homogeneous expansion we can define $\mathcal R^sf$ for any $s\in\mathbb C$ by
\[\mathcal R^sf(z)=\sum_{k=1}^\infty k^sf_k(z).\]
We define the Besov-Lipschitz space
$\Lambda^{p,q}_\alpha(\bn)=\Lambda^{p,q}_\alpha,$ $\alpha >0$, by
\begin{equation}\label{rs}
\begin{aligned}
 f\in  \Lambda^{p,q}_\alpha\Longleftrightarrow
\mathcal R^sf\in H ^{p,q,s-\alpha},
\end{aligned}
\end{equation}
where $s$ is the smallest integer  greater than $\alpha.$ In
Section 2 (Theorem 2.1) we give a new proof of the well known fact
that $s$ in \eqref{rs} can be replaced by any real number greater
than $\alpha .$ For $1\leq q <\infty $ the spaces
$\Lambda^{p,q}_{\alpha }$ are usually called Besov spaces. For the
Lipschitz space $\Lambda^{p,\infty }_{\alpha }$ we write
$\Lambda^p_\alpha=\Lambda_\alpha^{p,\infty}.$ In particular,
$\Lambda^{\infty}_\alpha$ consists of those $f\in H(\bn)$ for
which
$$
    \sup_{z\in \mathbb B}(1-|z|^2)^{s-\alpha}|\mathcal R^sf(z)|<\infty,
$$
where $s$ is as above. This condition has sense for $s=\alpha$ as
well and defines the Hardy-Sobolev spaces $H^{\infty }
_\alpha(\bn),$ $\alpha>0.$ More generally, the Hardy-Sobolev space
$H^p_{\alpha }(\bn )$ is defined by
$$
\begin{aligned} H^p_{\alpha
}=H^p_\alpha(\bn)=\{f\in H(\bn): \mathcal R^\alpha f\in
H^p(\bn)\}.
 \end{aligned}
 $$

\begin{prop}\label{replace}   In the definition of $\Lambda^{p,q}_\alpha,$ condition
\eqref{rs} can be replaced by
\begin{equation}\label{des}
\begin{aligned}
\nabla_s(f)\in L^{p,q}_{s-\alpha}(\bn),
 \end{aligned}
 \end{equation}
where $\nabla_s(f)$ is the $s$-th order gradient of $f$.
\end{prop} The second order gradient is defined as
$$
\begin{aligned}
\nabla_2(f)(z)=\bigg(\sum_{j,k=1}^N{|\partial_j\partial_kf(z)|^2}\bigg)^{1/2},
 \end{aligned}
 $$
where
\[\begin{aligned}  \partial_jf(z)=\frac{\partial f}{\partial z_j}(z). \end{aligned} \]
In the case $s>2$, $\nabla_s(f)$ is defined similarly. Of course
\(\nabla_1(f)(z)=|\nabla f(z)|,\) where $\nabla f(z)$ is the
ordinary holomorphic gradient.

The inequality
\begin{equation}\label{1.4}\begin{aligned}M_p(r,\mathcal R^sf)\le C\sum_{k=1}^{s-1}\nabla_k(f)(0)
+M_p(r,\nabla_s(f)),  \end{aligned} \end{equation}
which proves  that \eqref{des} implies \eqref{rs}  is simple.
 The reverse implication
  can also be verified in a relatively simple way (see, e.g.,
  the proofs of Lemmas 1 and 2 in \cite{17}).
However the following result, due to Ahern and Schneider \cite{1},
 shows that much more holds.

\begin{otherl}\label{ahern} If $f\in H(\bn)$, then $M_p(r,\nabla f)
\le CM_p(r,\mathcal Rf),for $
$1/4<r<1.$
 \end{otherl}

This fact is a reformulation of the original result of Ahern and
Schneider (see Rudin \cite[\$6.6.3]{19}). We will deduce the
inequality $M_p(r,\nabla_n(f))\le CM_p(r,\mathcal R^nf)$, from
Lemma \ref{ahern}, in Section 4  (Lemma \ref{ahern2}). In Section
4 we will also show that the radial derivative in $(1.2)$ can also
be replaced by the tangential gradients.

\subsection{Mean Besov-Lipschitz spaces}

Very recently, by using the $L^p$-modulus of continuity of order
one and two, Cho, Koo and Kwon \cite{5,13} and Cho and Zhu
\cite{6} defined the appropriate holomorphic mean Besov-Lipschitz
spaces. Their definition of the $L^p$-modulus of continuity can be
extended as follows:

Let $\mathcal U$ denote the group of all unitary transformations
of $\mathbb C^N.$ For $U\in \mathcal U$  and a function $f\in
H(\mathbb B_N),$ we let
$$
\begin{aligned} f_U(z)=f(Uz),
 \end{aligned}
 $$
 $$
 \begin{aligned}
 \Delta^1_Uf(z)=f(Uz)-f(z),
  \end{aligned}
  $$
$$
\begin{aligned} \Delta^n_Uf=\Delta^1_U(\Delta^{n-1}_Uf)\quad
(n\ge 2).
  \end{aligned}
  $$
  In particular
  $$
  \begin{aligned}
  \Delta^2_Uf(z)=f(U^2z)-2f(Uz)+f(z).
    \end{aligned}
 $$

Then, for $f\in H^p,$ let
\begin{equation}\label{}\begin{aligned}
\omega_n(\delta ,f)_p=\sup_{\|U-I\|<\delta,\, U\in\mathcal
U}\|\Delta_U^nf\|_p\quad (Iz=z).
 \end{aligned} \end{equation}

The holomorphic  mean Besov-Lipschitz space $\rm Lip^{p,q}_{\alpha
}(\bn )=\rm Lip^{p,q}_{\alpha }$ consists of those $f\in H^p(\bn)$
for which
$$
\int_0^1\biggl [\frac {\omega_n(\delta ,f)_p }{\delta^{\alpha
}}\biggr ]^q\frac {d\delta }{\delta }<\infty \quad \text{if}\quad
1\leq q<\infty
$$
and
$$
\omega_n(\delta ,f)_p=\mathcal O(\delta^{\alpha }), \quad
\text{if}\quad q=\infty ,
$$
where $n$ is the smallest integer greater than $\alpha $.

We consider other two moduli:
\begin{equation}\label{}\begin{aligned}
\omega_n^-(\delta,f)_p=\sup_{|t|<\delta}\|\Delta^n_tf\|_p,
  \end{aligned} \end{equation}
  where
$$
\begin{aligned} \Delta^n_tf=\Delta_t^1(\Delta^{n-1}_tf),\quad
\Delta^1_tf(z)=f(e^{it}z)-f(z), \quad t\in \mathbb R.
  \end{aligned}
$$
 To define the second one denote by $\mathcal L=\mathcal L(\bn)$
the semigroup of all linear operators from $\bn$ to $\bn$, and
then let

\begin{equation}
\omega_n^+(\delta,f)_p=\sup_{\|U-I\|<\delta,\ U\in\mathcal
L}\|\Delta^n_Uf\|_p,
\end{equation}

 where $\Delta^n_U$ is defined as above. From $(1.5)$, $(1.6)$ and $(1.7)$
   we have
\begin{equation}\label{8}\begin{aligned}
{\omega_n^-(\delta,f)_p}\le{ {\omega_n(\delta,f)_p}}\le
{{\omega_n^+(\delta,f)_p}}.
 \end{aligned} \end{equation}

We are now ready to state our main result which shows that
different $L^p$-moduli of continuity defined above give the same
mean Besov-Lipschitz spaces.

  \begin{thm} We have $\Lambda_\alpha^{p,q}(\mathbb B_N)\subset H^p(\mathbb B_N)$.
  For a function $f\in H^p$ and $0<\alpha<n$, the following conditions are equivalent:
\begin{itemize}
\item[\rm $(\mathcal R)$] $f\in \Lambda^{p,q}_\alpha$;\\[-1ex]
\item [\rm $(\Delta)$] $\displaystyle\left(\int_0^1\left[\frac{\|\Delta_t^nf\|_{p}}{t^\alpha}\right]^q\,
\frac{dt}t\right)^{1/q}\ <\infty $;
\item [\rm $(\omega^{-})$] $\displaystyle\left(\int_0^1\left[\frac{\omega_n^{-}(\delta,f)_p}{\delta^\alpha}\right]^q\,
 \frac{d\delta}\delta\right)^{1/q}\ <\infty;$\\[1ex]
\item [\rm $(\omega)$] $\displaystyle\left(\int_0^1\left[\frac{\omega_n(\delta,f)_p}{\delta^\alpha}\right]^q\,
 \frac{d\delta}\delta\right)^{1/q}\ <\infty;$\\[1ex]
\item [\rm $(\omega^{+})$] $\displaystyle\left(\int_0^1\left[\frac{\omega_n^{+}(\delta,f)_p}{\delta^\alpha}\right]^q\,
 \frac{d\delta}\delta\right)^{1/q}\ <\infty;$\\[1ex]
\end{itemize}
 \end{thm}

In the case $q=\infty $ we have

  \begin{thm} We have $\Lambda_\alpha^{p}(\mathbb B_N)\subset H^p(\mathbb B_N)$.
  For a function $f\in H^p$ and $0<\alpha<n$, the following conditions are equivalent:
\begin{itemize}
\item[\rm $(\mathcal R)$] $f\in \Lambda^{p}_\alpha$;\\[-1ex]
\item [\rm $(\Delta)$] $||\Delta_t^nf\|_{p}=\mathcal O(t^\alpha ), \, 0<t<1;$
\item [\rm $(\omega^{-})$]
$\omega_n^{-}(\delta,f)_p=\mathcal O(\delta^\alpha ),\quad 0<\delta <1;$\\[1ex]
\item [\rm $(\omega )$]
$\omega_n(\delta,f)_p=\mathcal O(\delta^\alpha ),\quad 0<\delta <1;$\\[1ex]
\item [\rm $(\omega^{+})$]
$\omega_n^{+}(\delta,f)_p=\mathcal O(\delta^\alpha ),\quad 0<\delta <1;$\\[1ex]
\end{itemize}
 \end{thm}

In the case $N=1$ Theorems 1.2 and 1.3 are known.
 In the case when $1\leq p\leq \infty $, $q=\infty$ and $0<\alpha <1$,
 (then $(\omega^{-})$$\Longleftrightarrow$$(\Delta)$ is clear),
    this theorem was proved
 by Hardy and Littlewood \cite{8}. In the case $p=q=\infty,$ $\alpha=1,$
 this was proved by Zygmund \cite{21}.
 The equivalence $(\mathcal R)$$\Longleftrightarrow$$(\omega^{-})$ for all $p,q,\alpha$
 was proved by Oswald \cite{14}, while
 $(\mathcal R)$$\Longleftrightarrow$$(\Delta)$ was proved in \cite{15,16}.
  (Of course the implication $(\omega^{-})$$\implies$$(\Delta)$
 is clear.) The paper \cite{4} of Blasco and De Souza is also relevant here.

Very recently the equivalence $(\mathcal
R)$$\Longleftrightarrow$$(\omega)$, for $0<\alpha <2$,
 was proved by Cho, Koo and Kwon \cite{5,13} and by Cho and Zhu \cite{6}.

Concerning the Hardy-Sobolev spaces we have the following result,
that was proved in \cite{16} for $N=1$.

\begin{thm}
Let $n$ be a positive integer. Then $H^p_n(\mathbb B_N)\subset
H^p(\mathbb B_N).$ Moreover a function $f\in H^p(\mathbb B_N)$
belongs to $H^p_n(\mathbb B_N)$ if and only if
\begin{equation}\label{hs}\begin{aligned}
\omega_n(\delta ,f)_p=\mathcal O(\delta^n),\quad 0<\delta<1.
  \end{aligned} \end{equation}
In \eqref{hs} $\omega_n(\delta ,f)_p$ may be replaced by
 $\omega_n^{-}(\delta ,f)_p$ or by
 $\omega_n^{+}(\delta ,f)_p$.
\end{thm}

 The implication $(\Delta)\implies (\mathcal R)$ in Theorem 1.2
 follows from  the next theorem  that will be proved in Section 5.
 In $(1.11)$ below just take $\psi (t)=t^{q(n-\alpha ) -1}$, $0<t<1$,
 where $0<\alpha <n.$

\begin{thm}\label{main}
Let $f\in H^p(\bn),$  and let $\psi\in L^1(0,1)$ be a non-negative
function such that
\begin{equation}
\begin{aligned} \psi(2x)\le K\psi(x), \quad
0<x<1/2,
 \end{aligned}
 \end{equation}
 where $K$ is a positive constant. Then
 \begin{equation}\label{}\begin{aligned}
 \int_0^1 M_p^q(r, \nabla_n(f))\psi(1-r)\,dr \le C\int_0^1[t^{-n}\|\Delta^n_tf\|_p]^q\psi(t)\,dt,
  \end{aligned} \end{equation}
  where $C$ depends only on $K$, $p,$ $q,$ $n,$ and $N.$
 \end{thm}

 The implication $(\Delta)\implies (\mathcal R)$ in Theorem 1.3
  follows from the estimate
 \begin{equation}\label{}
 \begin{aligned} M_p(r,\nabla_n(f))\le C(1-r)^{-n-1}\int_0^{1-r}||\Delta ^n_tf||_pdt,
  \end{aligned}
  \end{equation}
 which is a consequence of  Theorem 1.5. (See Section 5).

The implications $(\omega^+)\implies (\omega )\implies
(\omega^{-})\implies (\Delta )$ follow from \eqref{8} and the
definition of $(\omega^{-}).$ To finish the proofs of Theorem 1.2
and Theorem 1.3 it remains to prove  the implication $(\mathcal R
)\implies (\omega^{+}).$ This follows from the following theorem.

\begin{thm}\label{th-1}
 If $f\in H(\bn)$ and
\begin{equation}\label{}
    \int_0^1 (1-r)^{n-1}M_p(r,\mathcal R^nf)\,dr\ <\infty,
\end{equation}
then $f\in H^p,$ and
\begin{equation}\label{-}\begin{aligned}
\omega_n^+(\delta,f)_p\le C\int_{1-\delta}^1M_p(r,\mathcal
R^nf)(1-r)^{n-1}\,dr,\quad 0<\delta<1.
 \end{aligned} \end{equation}
\end{thm}

The proof that $(1.13)$ implies $f\in H^p$ will be given in
Section 2 (Corollary 2.5), and the proof  of $(1.14)$ will be
given in Section 5.

In Section 2 (Corollary 2.4) we prove that $\Lambda^{p,q}_{\alpha
}\subset H^p$. Since the condition $f\in H^p_n$ implies $(1.13)$,
we have $H^p_n\subset H^p$. To finish the proof of Theorem 1.4
apply $(1.12)$ and $(1.14)$.

In the case $N=1,$  Theorem 1.6 was proved in \cite[Theorem
2.2]{16}. Similar theorems appear in \cite{5} in the case where
$n=1,2.$ For example,  Theorems 5.2 and 5.3 of \cite{5}  give
$$
\begin{aligned} \omega_2(\delta,f)_p\le
C\delta^2M_p(1-\delta,\mathcal
R^2f)+C\int_{1-\delta}^1(1-r)M_p(r,\mathcal R^2f)\,dr.
 \end{aligned}
 $$
 Here we observe that the  summand $\delta^2M_p(1-\delta,\mathcal R^2f)$
 is not needed because,
  by the increasing property of the integral means,
 \[\begin{aligned} \delta^2M_p(1-\delta,\mathcal R^2f)\le 2\int_{1-\delta}^1
 (1-r)M_p(r,\mathcal R^2f)\,dr.  \end{aligned} \]
 As a consequence  we note a generalization of \cite[Theorem~ 5.3]{5}:
 \begin{cor}\label{}
 We have
 \begin{equation}\label{strong}\begin{aligned}
 \omega^+_n(\delta,f_{1-\delta})_p\le C\delta^nM_p(1-\delta,\mathcal R^nf),\quad 0<\delta<1.
  \end{aligned} \end{equation}
  \end{cor}
  \begin{proof}
  It follows from the Theorem 1.6 (relation $(1.14)$), applied to the function
   $f_\rho(z)=f(\rho z)$, and the
   increasing property of the integral means that
  $$
  \begin{aligned}
  \omega^+_n(\delta,f_\rho)_p\le C\delta^nM_p(\rho,\mathcal R^nf), \quad 0<\rho<1, \ 0<\delta<1.
   \end{aligned}
   $$
   Now \eqref{strong} is obtained by taking $\rho=1-\delta.$
   \end{proof}
   \begin{rem}\label{}
   We will use a weaker variant of \eqref{strong} in the proof of Theorem \ref{th-1}
    (see Lemma \ref{4.3}).
     \end{rem}

 In order to give a further  consequence of  Theorem \ref{th-1}, let $\phi>0$ be an
continuous  increasing function on the interval
 $(0,1]$ and define the spaces $\Lambda^{p}_{n,\phi}(\bn)$ and ${\rm Lip}^p_{n,\phi}(\bn)$
  in the following way:

 The space $\Lambda^{p}_{n,\phi}(\bn)$ consists of those $f\in H(\bn)$ for which
 \begin{equation}\label{}\begin{aligned} M_p(r,\mathcal R^nf)=\mathcal O\bigg(\frac{\phi(1-r)}
 {(1-r)^{n}}\bigg), \quad r\to 1^-. \end{aligned} \end{equation}
 The space ${\rm Lip}^p_{n,\phi}(\bn)$ consists of those $f\in H^p(\bn) $ for which
 \begin{equation}\label{}\begin{aligned}
 \omega_n^+(\delta,f)_p=\mathcal O(\phi(\delta)), \quad \delta\to 0^+.
  \end{aligned} \end{equation}
  The little ``oh" spaces $\lambda^{p}_{n,\phi}$ and ${\rm lip}^p_{n,\phi}$ are defined by replacing
   ``$\mathcal O$" with ``o". The condition $t^n=\mathcal O (\phi(t))$ guarantees that the space
   $\Lambda^{p}_{n,\phi}$ is of infinite dimension; in fact then $\Lambda^{p}_{n,\phi}$
   contains the Hardy-Sobolev space $H^p_n(\bn).$
  This condition should be strengthened to $t^n=o(\phi(t)),$ $t\to 0^+$, if we want
  $ \lambda^{p}_{n,\phi}(\bn)$ to be infinite-dimensional.
  \begin{cor}\label{}
  If the function $x\mapsto \phi(x)/x^\alpha$ increases on $(0,1]$ for some $0<\alpha <n,$ then
  $\Lambda^{p}_{n,\phi}(\bn)\subset {\rm Lip}^p_{n,\phi}(\bn)$ and $\lambda^{p}_{n,\phi}
  (\bn)\subset {\rm lip}^p_{n,\phi}(\bn).$
   \end{cor}
   \begin{proof}
   If $f\in \Lambda^{p}_{n,\phi},$ then $f\in \Lambda^p_\alpha$ (because
   $\phi(t)\le \phi(1)t^\alpha$) and hence
   $f\in H^p(\bn)$ (see Corollary 2.4). Next,
   \[\begin{aligned}
   \omega_n^{+}(\delta,f)_p&\le C\int_{0}^\delta\frac{\phi(x)}x\,dx\\&=
   C\int_0^\delta \frac{\phi(x)}{x^\alpha}x^{\alpha-1}\,dx\\&
   \le C\frac{\phi(\delta)}{\delta^\alpha}\frac{\delta^\alpha}\alpha,
    \end{aligned} \]
    and hence $f\in {\rm Lip}^p_{n,\phi}$. The rest is proved similarly.
    \end{proof}

Note that ${\rm Lip}^p_{n,\phi }(\bn )\subset \Lambda^p_{n,\phi
}(\bn ).$ Use $(1.12)$ and $(1.16).$

The above corollary can be generalized to integrated
Besov-Lipschitz spaces. Let  $\Lambda^{p,q}_{n,\phi}(\bn),$ $1\leq
q<\infty,$ denote the class of those $f\in H(\bn)$ for which
\begin{equation}\label{}\begin{aligned}
\int_0^1\bigg(M_p(r,\mathcal R^nf)\frac{(1-r)^n}{\phi(1-r)}\bigg)^q\frac{dr}{1-r}\ <\infty.
 \end{aligned} \end{equation}
   The space $ {\rm Lip}^{p,q}_{n,\phi}(\bn)$ is the subclass of $H^p(\bn)$ for which
   \begin{equation}\label{}\begin{aligned}
   \int_0^1\bigg(\frac{\|\Delta^n_t f\|_p}{\phi(t)}\bigg)^q\frac{dt}{t}\ <\infty.
     \end{aligned} \end{equation}
     Suppose
     \begin{equation}\label{}\begin{aligned}
     \int_0^1\bigg(\frac{t^n}{\phi(t)}\bigg)^q\,\frac{dt}t\ <\infty,
      \end{aligned} \end{equation}
      which implies that ${\Lambda}^{p,q}_{n,\phi}(\bn)$
      is infinite-dimensional (in fact, it contains all polynomials).
If $(1.20)$ is satisfied then ${\rm Lip}^{p,q}_{n,\phi }(\bn
)\subset \Lambda^{p,q}_{n,\phi }(\bn ).$ This fact is an easy
consequence of Theorem 1.5. (Use (1.11),(1.18) and (1.19).)

 \begin{cor}\label{}
      If the function $x\mapsto \phi(x)/x^\alpha$ increases on $(0,1]$ for some $\alpha>0,$ then
      $\Lambda^{p,q}_{n,\phi}(\bn)\subset {\rm Lip}^{p,q}_{n,\phi}(\bn).$
\end{cor}
   \begin{proof}
   Since $\phi(t)\le \phi(1)t^\alpha,$ we have $\Lambda^{p,q}_{n,\phi}\subset \Lambda^{p,q}_{\alpha}\subset H^p$.

Let $f\in \Lambda^{p,q}_{n,\phi }$. We  prove that $f\in {\rm
Lip}^{p,q}_{n,\phi }.$ By Jensen's inequality,
$$
\biggl (\int_{1-t}^1(1-r)^{n-1 }M_p(r,\mathcal R^nf)\frac {\alpha
dr}{ t^{\alpha }}\biggr )^q\leq \int_{1-t}^1(1-r)^{q(n-\alpha
)+\alpha -1}M_p(r,\mathcal R^nf)^q\frac {\alpha dr}{ t^{\alpha }}.
$$
From this and $(1.14)$ we have
$$
\begin{aligned}
&\int_0^1\biggl (\frac {||\Delta^n_tf||_p}{\phi (t)}\biggr
)^q\frac {dt}{t}\leq \int_0^1\biggl (\frac
{\omega^{+}_n(t,f)_p}{\phi (t)}\biggr )^q\frac {dt}{t}\\
&\leq C\int_0^1\frac {t^{\alpha q-\alpha -1}}{(\phi (t))^q}\biggl
( \int_{1-t}^1(1-r)^{q(n-\alpha )+\alpha -1}M_p(r,\mathcal
R^nf)^qdr\biggr )dt\\
&=C\int_0^1(1-r)^{q(n-\alpha )+\alpha -1}M_p(r,\mathcal
R^nf)^q\biggl (\int_{1-r}^1\frac {t^{\alpha q -\alpha -1}}{\phi
(t)^q}dt\biggr )dr\\
\end{aligned}
$$
Using the inequality $\frac {t^{\alpha }}{\phi (t)}\leq \frac
{(1-r)^{\alpha }}{\phi (1-r)},$ $1-r\leq t$, one shows that the
inner integral is dominated by $\frac {C(1-r)^{\alpha q-\alpha
}}{(\phi (1-r))^q},$ which completes the proof.

\end{proof}

\section{Finite-dimensional decomposition and applications}

Let $\psi:(0,\infty)\mapsto \mathbb C$ be a $C^\infty$ function with
 compact support in $(0,\infty).$ Consider the polynomials
$$
\begin{aligned}
\Omega_\nu(z)=\Omega_\nu^\psi(z)=\sum_{j=0}^\infty
\psi\bigg(\frac{j}{2^{\nu-1}}\bigg)z^j, \quad z\in \mathbb B_1,
\quad \nu\ge 1.
\end{aligned}
$$
 It is proved in  \cite[Theorem 7.3.4]{18} that
  \begin{equation}\label{elem}\begin{aligned}
 \|\Omega_\nu*f\|_p\le C\|f\|_p,\quad f\in H^p(\mathbb B_N),
  \end{aligned} \end{equation}
  in the case $N=1.$ (Theorem 7.3.4 of \cite{18} says much more than we need. Inequality~
   \eqref{elem}  is rather elementary, see \cite[Lemma 7.3.2]{18} or \cite{10}.)
   The Hadamard product of a functions $f\in H^p(\bn)$ and $g\in H^p(\mathbb B_1)$ is defined as
   \[g*f(z)=\sum_{k=0}^\infty \hat g(k)f_k(z),\]
   where $f=\sum_{k=0}^{\infty } f_k$ is the homogeneous expansion of $f$.

Integration by slices shows that \eqref{elem} holds for all $N.$
Using this one can choose
 $\psi (t)=\omega (t/2)-\omega (t)$, where  $\omega (t)$ is any
infinitely differentiable function with $\omega (t)=1$ for $t\leq
1$, $0\leq \omega (t)\leq 1$ for $1<t\leq 2$ and $\omega (t)=0$
for $t>2$,
 so that the polynomials $\Omega_\nu^\psi=:W_\nu$ satisfies:
$$
\begin{aligned}
 \mathop{\rm supp}\widehat W_\nu\subset
[2^{\nu-1},2^{\nu+1}), \quad \nu\ge 1,
  \end{aligned}
  $$
\begin{equation}\label{from}\begin{aligned} f(z)=\sum_{\nu=0}^\infty W_\nu*f(z),
\quad f\in H(\mathbb B_N), \end{aligned} \end{equation}
where
\[W_0(z)=1+z,\]
and
\begin{equation}\label{2.5}\begin{aligned}
\|W_\nu*f\|_p\le C\|f\|_p, \quad f\in H(\mathbb B_N).
 \end{aligned} \end{equation}
(See, e.g. \cite{11}.)

We define  $\mathcal B^{p,q}_\beta = \mathcal B^{p,q}_\beta
(\bn),$ $-\infty<\beta<\infty,$ to be the space of those $f\in
H(\mathbb B_N)$ for which
\begin{equation}\label{}\begin{aligned}
\|f\|_{\mathcal B^{p,q}_\beta}:=\|\{2^{-\nu\beta}\|W_\nu*f\|_p\}_{\nu=0}^\infty\|_{\ell^q}<\infty.
 \end{aligned} \end{equation}

The following theorem, for $N=1$, is proved in \cite{12}. A
similar argument shows that it holds for all $N\geq 1.$ See
\cite[Lemma~ 2.2]{11}.

\begin{other}\label{bergman} We have $H^{p,q,\alpha }(\bn)=\mathcal B^{p,q}_\alpha (\bn).$
 \end{other}

 \begin{thm}\label{besov} We have $\Lambda^{p,q}_\alpha=\mathcal B^{p,q}_{-\alpha}$.
 More generally, if $s>\alpha$ is any real number, then
  $$
  \begin{aligned}
 \{f\in H(\bn): \|\mathcal R^sf\|_{p,q,s-\alpha}<\infty\}=\mathcal B^{p,q}_{-\alpha},
   \end{aligned}
   $$
   and consequently
   \begin{equation}\label{zhu-cho3.11}\begin{aligned}
   \Lambda^{p,q}_\alpha= \{f\in H(\bn): \|\mathcal R^sf\|_{p,q,s-\alpha}<\infty\}.
     \end{aligned} \end{equation}
 \end{thm}

 \begin{rem}\label{}
 Relation \eqref{zhu-cho3.11} is proved in \cite[Theorem 3.11]{6} in a different manner.
  \end{rem}
\begin{proof}
 This is an immediate consequence of Theorem \ref{bergman} and the following lemma. \end{proof}
 \begin{lem}\label{WR}
 If $f\in H(\mathbb B_N)$,  and $s\in\mathbb C,$ then
 \begin{equation}\label{}\begin{aligned} \|W_\nu*\mathcal R^sf\|_p\asymp 2^{\nu\mathop{\rm Re}s}\|W_\nu*f\|_p,
  \quad \nu\ge2. \end{aligned} \end{equation}
   \end{lem}

   The notation $a\asymp b$ means that $a/b$ lies between two positive constants. In this case these
    constants are independent of $\nu$ and $f.$
 \begin{proof} Let
 $$
 \begin{aligned}
 Q_\nu=W_{\nu-1}+W_\nu+W_{\nu+1},\quad \nu\ge 0,
  \end{aligned}
  $$
  where $W_{-1}:=0.$
  Since $Q_\nu*W_k=0$ for $|\nu-k|\ge 2,$ it follows from \eqref{from} that
  $$
  \begin{aligned} Q_\nu*W_\nu=W_\nu.
   \end{aligned}
   $$

  Let
  $$
  \begin{aligned}
  \varphi_s(x)=x^s(\psi(x/2)+\psi(x)+\psi(2x)).
    \end{aligned}
$$
  We have, for $\nu\ge 2,$
  $$
  \begin{aligned} Q_\nu*\mathcal R^sf=2^{(\nu-1)s}\Omega_\nu^{\varphi_s}*f.
   \end{aligned}
   $$
  From this and $(2.1)$ we get
  \[\|Q_\nu*\mathcal R^sf\|_p\le C2^{\nu\mathop{\rm Re}s}\|f\|_p.\]
  Replacing $f$ by $W_\nu*f$ we get
  $$
    \begin{aligned}
    \|W_\nu*\mathcal R^sf\|_p&=\|Q_\nu*\mathcal R^s(W_\nu*f)\|_p\\&
    \le C2^{\nu\mathop{\rm Re}s}\|W_\nu*f\|_p.
    \end{aligned}
  $$
  The reverse inequality now follows from the relation
  \[W_\nu*f=W_\nu*\mathcal R^{-s}(\mathcal R^sf).\]
  \end{proof}

  \begin{cor} We have
  $\Lambda^{p,q}_\alpha(\bn)\subset H^p(\bn).$
  \end{cor}
  \begin{proof}
  Let $f\in \Lambda ^{p,q}_\alpha(\bn).$ Then, by Theorem \ref{besov},
  $\|W_\nu*f\|_p\le C2^{-n\alpha}.$ It follows that
  \[\|f\|_p\le \sum_{\nu=0}^\infty \|W_\nu*f\|_p\ <\infty.\]
    \end{proof}
\begin{cor}\label{}
If $f\in H(\bn)$ and \[K:=\int_0^1(1-r)^{n-1}M_p(r,\mathcal R^nf)\,dr<\infty,\]
then $f\in H^p,$ and $\|f\|_p\le C( |f(0)|+K^{1/p}).$
\end{cor}
\begin{proof}
by Theorem \ref{bergman} and Lemma \ref{WR}, we have
\[K+|f(0)|\asymp \sum_{\nu=0}^\infty\|W_\nu*f\|_p.\]
The result follows.
 \end{proof}
 \subsection{The operators $\mathcal R^{s,t}$}
 Let $s,t$ be real numbers such that neither $N+s$ nor $N+s+t$ is a negative integer. Let
 $$
 \begin{aligned}
 \mathcal R^{s,t}f(z)=\sum_{k=0}^\infty\frac{\Gamma(N+1+s)\Gamma(N+1+k+s+t)}{\Gamma(N+1+s+t)\Gamma(N+1+k+s)}
 f_k(z).
  \end{aligned}
  $$
  The following theorem is proved in \cite[Theorem 3.10]{6}.
  \begin{other}\label{in}
  Suppose $t>\alpha,$ $f\in H^p.$ If $s$ is a real number such
  that neither $N+s$ nor $N+s+t$
  is a negative integer, then $f\in\Lambda^{p,q}_\alpha$ if and only
  if $\mathcal R^{s,t}f\in L^{p,q}_{t-\alpha}.$
   \end{other}

   An application of Stirling's formula shows that
   $$
   \begin{aligned}
   \frac{\Gamma(N+1+s)\Gamma(N+1+k+s+t)}{\Gamma(N+1+s+t)\Gamma(N+1+k+s)}=k^t\Big(a_1+\frac {a_2}k+
   \mathcal O\big(\frac 1{k^2}\big)\Big), \quad k\to\infty,
    \end{aligned}
    $$
    where $a_1\ne 0$ and $a_2$  are constants. Therefore Theorem \ref{in} is a
     consequence of the following result.
    \begin{thm}\label{}
Let $T:H(\bn)\mapsto H(\bn)$ be an operator of the form
$$
\begin{aligned} Tf(z)=\sum_{k=0}^\infty \lambda_kf_k(z),
 \end{aligned}
 $$
 where
 $$
 \begin{aligned}
 \lambda_k=k^t\Big(a_1+\frac {a_2}k+ o\big(\frac 1{k}\big)\Big), \quad k\to \infty,
   \end{aligned}
  $$
   and $t$, $a_1\ne 0,$ and $a_2$ are constants. Let $t>\alpha.$ Then $f\in \Lambda^{p,q}_\alpha$ if and only if
   $Tf\in L^{p,q}_{t-\alpha}.$
     \end{thm}
     \begin{proof}
     By Theorem \ref{bergman} and Lemma \ref{WR}, it suffices to prove that
     $$
     \begin{aligned}
     \|T(W_\nu*f)\|_p\asymp \|W_\nu*f\|_p,\quad \nu\to\infty,
       \end{aligned}
   $$
       under the hypothesis $t=0$ and $a_1=1.$ Let $Q=W_\nu*f$.
       Let $$T_1f(z)=\sum_{k=1}^\infty \frac{\eta_k}kf_k(z),$$
       where $\eta_k\to0$ as  $k\to\infty.$ Let $Q_k(z)=\widehat W_\nu(k)f_k(z).$
       Then
       \[TQ=Q+a_2\mathcal R^{-1}Q+T_1Q.\]
       Hence, by Lemma \ref{WR} and the inequality
       \( \|Q_k\|_p\le \|Q\|_p \),
       \[\begin{aligned} \|TQ\|_p&\le \|Q\|_p+C_1|a_2|2^{{-\nu}}\|Q\|_p +
       2^{1-\nu}\sum_{k=2^{\nu-1}}^{2^{\nu+1}}|\eta_k|\,\|Q_k\|_p\\&
      \le  \|Q\|_p+C_1|a_2|2^{{-\nu}}\|Q\|_p +
       2^{1-\nu}\|Q\|_p\sum_{k=2^{\nu-1}}^{2^{\nu+1}}|\eta_k|\\&\le C\|Q\|_p.
       \end{aligned} \]
     In the other direction we have
     \[\begin{aligned} \|TQ\|_p&\geq \|Q\|_p-C_1|a_2|2^{{-\nu}}\|Q\|_p -
       2^{1-\nu}\sum_{k=2^{\nu-1}}^{2^{\nu+1}}|\eta_k|\,\|Q_k\|_p\\&
      \geq  \|Q\|_p-C_1|a_2|2^{{-\nu}}\|Q\|_p -
       2^{1-\nu}\|Q\|_p\sum_{k=2^{\nu-1}}^{2^{\nu+1}}|\eta_k|\\.  \end{aligned} \]
       Now choose $\nu_0$ so that $C_1|a_2|2^{{-\nu}}<1/4$ and $2^{1-\nu}
       \sum_{k=2^{\nu-1}}^{2^{\nu+1}}|\eta_k|<1/4$ for $\nu>\nu_0$
       to get $\|TQ\|_p\ge (1/2)\|Q\|_p$.  This proves the theorem.
                   \end{proof}
 \subsection{The case $1<p<\infty$} In this case the above discussion can be made
  simpler by using the Riesz projection theorem. Namely, then we can replace $W_\nu$ by
 $$
 \begin{aligned}
 V_\nu(z)=\sum_{k=2^{\nu-1}}^{2^{\nu}-1}z^k, \quad z\in\mathbb B_1, \ \nu\ge1,
  \end{aligned}
  $$
 and $V_0=1$. Obviously, these polynomials satisfy \eqref{from} and also, by the projection theorem,
  \eqref{2.5} $(1<p<\infty).$

\section{Best approximation by polynomials}

For a function $f\in H^p(\bn),$ let
\begin{equation}\label{}\begin{aligned}
E_\nu(f)_p=\inf\{\|f-P\|_p:P\in \mathscr P_\nu(\bn)\},
 \end{aligned} \end{equation}
 where $\mathscr P_\nu(\bn)$ is the subset of $H(\bn)$ consisting of all polynomials
  of degree $\le \nu.$
 The following characterization of the one variable Besov-Lipschitz spaces is well known
 (for a proof see \cite{9}).
  \begin{thm}\label{}
  A function $f\in H^p(\bn)$ is in  $\Lambda^{p,q}_\alpha$ if and only if the sequence
  $\{2^{\nu\alpha}E_{2^\nu}(f)_p\}_{\nu=0}^\infty$ is in $\ell^q.$
\end{thm}
  \begin{proof}
  Let $P_\nu$ be a sequence of polynomials of degree $\le 2^\nu$ such that
  \[\{2^{\nu\alpha}\|f-P_\nu\|_p\}_0^\infty\in \ell^q.\]
  Since \[(f-P_\nu)*W_{\nu+2}=f*W_{\nu+2}\]
  we have
  \[\|W_{\nu+2}*f\|_p\le C\|f-P_\nu\|_p.\]
  This and Theorem \ref{besov} show that  $f\in \Lambda^{p,q}_\alpha.$ To prove
   the converse observe that
  $Q_\nu=\sum_{k=0}^{\nu-1}W_k,$ $\nu \ge1,$ is a polynomial of degree $\le 2^\nu,$ and therefore
by using $(2.2)$ we obtain
  \[\begin{aligned} E_{2^\nu}(f)_p&\le \|f-Q_\nu * f\|_{p}\\&\le\sum_{k={\nu-1}}^\infty\|W_k*f\|_p.
  \end{aligned} \]
  Hence
  $$
  \begin{aligned} \|\{2^{\nu\alpha}E_{2^\nu}(f)_p\}_{\nu=0}^\infty\|_{\ell^q}
   \le \|\{2^{\nu\alpha}s_{\nu-1}\}_{\nu=1}^\infty\|_{\ell^q},
   \end{aligned}
   $$
   where\[s_\nu=\sum_{k=\nu}^\infty \|W_k*f\|_p.\]
   Now the desired result follows Theorem \ref{besov} and the following lemma. \end{proof}
   \begin{lem}\label{} If $\{s_\nu\}_0^\infty$ is a sequence of complex numbers such that
   $\{2^{\nu\alpha}|s_\nu-s_{\nu-1}|\}_1^\infty\in \ell^q$, then
    $\{2^{\nu\alpha}|s_\nu|\}_1^\infty\in \ell^q$.
     \end{lem}
     \begin{proof}  Assuming, as we may,  $s_\nu$ is eventually zero, we have
   \[\begin{aligned}
   M:=\|\{2^{\nu\alpha}s_{\nu-1}\}_{\nu=1}^\infty\|_{\ell^q}&
   \le\|\{2^{\nu\alpha}|s_{\nu-1}-s_\nu|\}_{\nu=1}^\infty\|_{\ell^q}+
   \|\{2^{\nu\alpha}s_{\nu}\}_{\nu=1}^\infty\|_{\ell^q}\\&=
   \|\{2^{\nu\alpha}|s_{\nu-1}-s_\nu|\}_{\nu=1}^\infty\|_{\ell^q}+
   \|\{2^{(\nu-1)\alpha}s_{\nu-1}\}_{\nu=2}^\infty\|_{\ell^q}
  \\& \le \|\{2^{\nu\alpha}|s_{\nu-1}-s_\nu|\}_{\nu=1}^\infty\|_{\ell^q}+
   2^{-\alpha}\|\{2^{\nu\alpha}s_{\nu-1}\}_{\nu=2}^\infty\|_{\ell^q}
   \\& \le \|\{2^{\nu\alpha}|s_{\nu-1}-s_\nu|\}_{\nu=1}^\infty\|_{\ell^q}+2^{-\alpha}{ M}.
     \end{aligned} \]
     Since $M$ is finite we get
     $$
     \begin{aligned}
     M\le (1-2^{-\alpha})^{-1}\|\{2^{\nu\alpha}|s_{\nu-1}-s_\nu|\}_{\nu=1}^\infty\|_{\ell^q}.
      \end{aligned}
      $$
      This proves the lemma.
  \end{proof}

\section{Characterizations of Besov-Lipschitz spaces}

Note that another way to express $\mathcal R$ is
\begin{equation}\label{theta}\mathcal Rf(re^{i\theta}\zeta)=-i\frac{\partial }
{\partial \theta}f(re^{i\theta}\zeta), \quad |\zeta|=1,\ 0\le
r<1.\end{equation}

As it was noticed in the introduction Proposition 1.1 follows from
the inequality  \eqref{1.4} and the following inequality:

\begin{lem}\label{ahern2} $\|\nabla_n(f)\|_p\le C\|\mathcal R^nf\|_p$  \end{lem}
\begin{proof} In the case $n=1$ this is just Lemma \ref{ahern}. Let $n\ge2.$ We have to prove that
\[\|\partial _{j_1}\partial _{j_2}\ldots \partial _{j_n}f\|_p\le C\|\mathcal R^nf\|_p,\]
where $1\le j_1,j_2,\dots, j_n\le N.$ By induction hypothesis, we
have
\[\|\partial _{j_1}\partial _{j_2}\ldots \partial _{j_n}f\|_p\le C\|\mathcal R^{n-1}\partial _{j_n}f\|_p. \]
On the other hand, it is easy to see (by induction) that
\begin{equation}\label{}
    \mathcal R^{n-1}\partial_ jf =\partial_j(\mathcal R-I)^{n-1}f,
\end{equation}
whence, by Lemma \ref{ahern},
\[\begin{aligned}\|\mathcal R^{n-1}\partial _{j_n}f\|_p&\le \sum_{k=0}^{n-1}\binom {n-1}k\|\partial_{j_n}
\mathcal R^{k}f\|_p\\&\le C \sum_{k=0}^{n-1}\binom {n-1}k\|\mathcal R^{k+1}f\|_p.
\end{aligned}\]
    Now we prove that
    \begin{equation}\label{r^2}\begin{aligned} \|\mathcal Rf\|_p\le \|\mathcal R^2f\|_p, \end{aligned} \end{equation}
    which will conclude the proof.
    First observe that \eqref{r^2} reduces to the case $N=1$ by using integration by slices. We also may assume
    that $f$ is a polynomial. Then
    $$
    \begin{aligned}
    \mathcal Rf(\zeta)=\int_0^1\frac1r \mathcal R^2f(r\zeta)\,dr, \quad |\zeta|=1.
     \end{aligned}
     $$
     Hence
     $$
     \begin{aligned}
     \|\mathcal Rf\|_p\le \int_0^1\frac 1rM_p(r,\mathcal R^2f)\,dr=\int_0^1M_p(r,g)\,dr,
      \end{aligned}
      $$
      where
      \[g(r\zeta)=\sum_{k=1}^\infty k^2\hat f(k) r^{k-1}\zeta^{k-1}.\]
      Since \[\begin{aligned} M_p(r,g)\le \|g\|_p=\|\mathcal R^2 f\|_p,   \end{aligned} \]
      the proof of \eqref{r^2} and of the lemma is completed.
          \end{proof}

For the characterizations of Besov-Lipschitz spaces that will be
given below we use tangential derivatives. For $1\leq i,j\leq N$
the tangential derivatives $T_{i,j}$ and $\overline{T}_{i,j}$ are
defined by
$$
T_{i,j}=\bar{z}_{i}\frac {\partial }{\partial z_j}-\bar{z}_j\frac
{\partial }{\partial z_i}\quad \text{and}\quad
\overline{T}_{i,j}=z_i\frac {\partial }{\partial
\bar{z}_j}-z_j\frac {\partial }{\partial \bar{z}_i}.
$$

We will consider operators $T=T_1\cdots T_k$ where the
$T_1,...,T_k$ are chosen among either the $T_{i,j}$ or the
$\overline{T}_{i,j}$, $1\leq i,j\leq N.$ We define by $\{T_{\delta
}^{+}\}$, $\delta \in C_k^{+}$ the collection of such operators
and define
$$
\nabla^{k}_{T^{+}}f(z)=\sum_{\delta \in C^{+}_k}|T^{+}_{\delta
}f(z)|.
$$

If $T_1,...,T_k$ are chosen among the $T_{i,j}$ (not the
$\overline{T}_{i,j}$'s ), the collection of such $T=T_1\cdots T_k$
will be denoted by $\{T_{\delta }\}$, $\delta \in C_k$ and the
complex-tangential gradient as
$$
\nabla^k_Tf(z)=\sum_{\delta \in C_k}|T_{\delta }f(z)|.
$$

We recall the definition of the non-isotropic weight of a
differential operator. We assign weight $1$ to $\mathcal R$ while
$T_{i,j}$ and $\overline{T}_{i,j}$ are given weight $1/2$ each. We
will consider differential operators
\begin{equation}
Xf=X_1\cdots X_kf,
\end{equation}
where each $X_j$ is $\mathcal R$ or the one of $T_{i,j}$ or
$\overline{T}_{i,j}.$ For such an operator its weight is defined
to be the sum of each weights of $X_j.$

For $z\in \bn$ and $\delta >0$ let $P(z,\delta )$ be the
non-isotropic polydisc defined as follows. If $z=r\xi $, $0\leq
r<1$, $\xi \in S_N$, pick $\eta_2,...,\eta_N$ so that $\{\xi
,\eta_2,...,\eta_N\}$ is an orthonormal basis of $\cn$. Then
$$
P(z,\delta )=\{w=r\xi +\lambda \xi
+\sum_{j=2}^N\lambda_j\eta_j,\quad |\lambda |<\delta ,
|\lambda_j|<\sqrt{\delta }, j=2,3,...,N \}
$$
The following is a weak version of Lemma 2.5([2]).

\begin{lem}
Let $X$ and $Y$ be the differential operators of the form $(4.4)$
with weight of $X$ being $m$. If $f\in H(\bn)$, then we have
$$
|XYf(z)|^p\leq \frac {C}{\delta^{N+1+mp}}\int_{P(z,\delta
)}|Yf(w)|^pdv(w),
$$
for $P(z,\delta )\subset \bn.$
\end{lem}

As a corollary we have that
\begin{equation}
M_p(r,XYf)\leq \frac {C}{(1-r)^m}M_p(\bar{r},Yf),
\end{equation}
where $1/2<r<1$ and $\bar{r}=r+\frac {1-r}{4}.$

We will also need the following two lemmas.

\begin{lem}(\cite{2})
Let $\alpha >0$, $\beta >0$ and $1\leq p<\infty .$ Then we have
\begin{equation}
\int_0^1(1-r)^{\alpha -1}\biggl (\int_0^r(r-t)^{\beta
-1}F(t)dt\biggr )^pdr\leq C\int_0^1(1-r)^{\alpha +\beta
p-1}F(r)^pdr,
\end{equation}
for all $F\geq 0.$
\end{lem}

\begin{lem}(\cite{5})
Let $m$ and $k$ be positive integers. If $f\in H(\bn)$, then for
$T^{+}_{\delta }f$, $\delta \in C^{+}_k$, and $1/2<r<1$, we have
$$
M_p(r,T^{+}_{\delta }f)\leq C\biggl
(\sup_{|z|<1/2}|f(z)|+\int_0^r(r-t)^{m-1}M_p(t,T^{+}_{\delta
}\mathcal R^mf)dt\biggr ).
$$
\end{lem}

Now we are ready to give the characterization of Besov-Lipschitz
spaces that involve the tangential derivatives.

\begin{thm}
Suppose that $k>2\alpha $, $k$ is an integer. If $f\in H^p$ then
the following statements are equivalent:
\item{\rm (i)} \quad $f\in \Lambda^{p,q}_{\alpha }$;
\item{\rm (ii)}\quad $||\nabla^k_Tf||_{p,q,k/2-\alpha }<\infty $;
\item{\rm (iii)}\quad $||\nabla^k_{T^{+}}f||_{p,q,k/2-\alpha
}<\infty $.
\end{thm}

The equivalence $(i)\Longleftrightarrow (iii)$, for $0<\alpha <2$,
was considered in \cite{5,13}.

\begin{proof}
Obviously, $(iii)\Rightarrow (ii).$ Now we show that, for $1\leq
q<\infty $, $(i)\Rightarrow (iii).$ Let $f\in
\Lambda^{p,q}_{\alpha }.$ Assume that $m$ is a positive integer
such that $0<\alpha <m.$ By Theorem 2.1
$$
\int_0^1(1-r)^{q(m-\alpha )-1}M_p(r,\mathcal R^mf)^qdr<\infty .
$$

By using Lemma 4.4, $(4.6)$ and $(4.5)$ we find for $\delta \in
C_k^{+}$
$$
\begin{aligned}
&||T^{+}_{\delta }f||_{p,q,k/2-\alpha }\leq  C\biggl
(\sup_{|z|<1/2}|f(z)|\\
&+\biggl (\int_0^1(1-r)^{q(k/2-\alpha )-1}\biggl
(\int_0^r(r-t)^{m-1}M_p(t,T^{+}_{\delta }\mathcal R^mf)^qdt\biggr
)^qdr\biggr )^{1/q}\biggr )\\
&\leq C\biggl (\sup_{|z|<1/2}|f(z)|+\biggl
(\int_0^1(1-r)^{q(k/2+m-\alpha )-1}M_p(r,T^{+}_{\delta
}\mathcal R^mf)^qdr\biggr )^{1/q}\biggr )\\
&\leq C\biggl (\sup_{|z|<1/2}|f(z)|+\biggl
(\int_0^1(1-r)^{q(m-\alpha
)-1}M_p(r,\mathcal R^mf)^qdr\biggr )^{1/q}\biggr )\\
\end{aligned}
$$
From this it follows that $||\nabla^k_{T^{+}}f||_{p,q,k/2-\alpha
}<\infty .$

The implication $(i)\Rightarrow (iii)$ holds also for $q=\infty $.
Use Lemma 4.4 and $(4.5).$

$(ii)\Rightarrow (i)$

A simple calculation shows that there are constant $d_j$,
$j=0,1,...,k,$ such that
\begin{equation}
\sum_{j=0}^kd_j\mathcal R^{k-j}f=\sum_{\delta \in
C_k}\overline{T}_{\delta }T_{\delta }f.
\end{equation}
Here, if $T_{\delta }=T_1\cdots T_k$, where $T_m$, $1\leq m\leq
k$, are chosen among $T_{i,j}$'s, then $\overline{T}_{\delta
}=\overline{T}_1\cdots \overline{T}_{k}.$ By using Lemma 4.2 we
find that
\begin{equation}
M_p^p(r,\overline{T}_{\delta }T_{\delta })\leq \frac
{C}{(1-r)^{1+kp/2}}\int_{L_r}|T_{\delta }f(w)|^pdv(w),
\end{equation}
where
$$
L_r=\{z: \frac {1-r}{2}<1-|z|<2(1-r) \}.
$$
From  $(4.7)$ and $(4.8)$ we conclude that
$$
\int_0^1(1-r)^{q(k-\alpha )-1}M_p(r, \sum_{j=0}^kd_j\mathcal
R^{k-j})^qdr \leq C\int_0^1(1-r)^{q(k/2-\alpha
)-1}M_p(r,\nabla^k_Tf)^qdr.
$$
From this it follows that $f\in \Lambda^{p,q}_{\alpha }.$ From
$(4.8)$ it follows that if $||\nabla^k_Tf||_{p,\infty ,k/2-\alpha
}<\infty $, then $f\in \Lambda^p_{\alpha }.$

\end{proof}

\section{Moduli of smoothness and mean growth of derivatives}\label{moduli}

In this section we prove Theorem 1.5 and Theorem 1.6.

\textit{Proof of Theorem 1.5}

Assume that $f$ is holomorphic in a neighborhood of the closed
ball. For  fixed $r\in (0,1)$ and $\zeta\in\sn,$ let
$h(\theta)=f(re^{i\theta}\zeta).$ By induction,
$$
\begin{aligned}
\Delta^n_th(\theta)=\int_{[0,t]^n}h^{(n)}(\theta+x_1+\ldots+x_n)\,dx_1\ldots
x_n.
 \end{aligned}
 $$
 Hence, by using $(4.1)$ we get
 \[\begin{aligned} i^n\mathcal R^nf(re^{i\theta}\zeta)t^n&=h^{(n)}(\theta)t^n \\&=\Delta^n_th(\theta)-
 \int_{[0,t]^n} (h^{(n)}(\theta+x_1+\ldots+x_n)-h^{(n)}(\theta))\,dx_1\dots x_n.\end{aligned} \]
 This implies that
 \[\begin{aligned}
 |\mathcal R^nf(re^{i\theta}\zeta)|t^n&\le |\Delta^n_tf(re^{i\theta} \zeta)|\\&\quad+\int_{[0,t]^n}\sup_{0<y<nt}
 |h^{(n+1)}(\theta +y)|(x_1+\ldots+x_n)\,dx_1\dots dx_n\\&
 =|\Delta^n_tf(re^{i\theta} \zeta)| +\frac n2 t^{n+1}\sup_{0<y<nt}|\mathcal R^{n+1}f(re^{i(\theta +y)}\zeta)|.
  \end{aligned} \]
  Hence, by Minkowski's inequality and the complex maximal theorem,
   $$
   \begin{aligned}
   t^nM_p(r, \mathcal R^nf)\le \|\Delta^n_t f_r\|_p+C_0t^{n+1}M_p((3r+1)/4,\mathcal R^{n+1}f)
   \end{aligned}
   $$
   provided $0<t\le 1-r.$ (See Lemma 3.1 of \cite{16}.)
   From this, the inequality
   \[\begin{aligned}
   M_p((3r+1)/4,\mathcal R^{n+1}f)\le C_1(1-r)^{-1}
   M_p((1+r)/2,\mathcal R^{n}f),
   \end{aligned} \]
   and the inequality $\|\Delta_t^nf_r\|_p\le\|\Delta_t^nf\|_p,$
   it follows that
   \begin{equation}\label{multip}\begin{aligned}
   M_p(r, \mathcal R^nf)\le t^{-n}\|\Delta^n_t f\|_p+C_0C_1t(1-r)^{-1}M_p((1+r)/2,\mathcal R^{n}f).
   \end{aligned} \end{equation}
    Now let
    $q<\infty$ and
    $$
    \begin{aligned}
     A(r)=M_p^q(r,\mathcal R^nf)\psi(1-r).
     \end{aligned}
     $$
     Since $\psi(r/2)\ge (1/K)\psi(r),$  we have
     \[\begin{aligned}
     A((1+r)/2)&=M_p^q((1+r)/2,\mathcal R^nf)\psi((1-r)/2)\\&\ge
     (1/K)M_p^q((1+r)/2,\mathcal R^nf)\psi(1-r)
     \end{aligned} \]
     and therefore by using this and \eqref{multip}  we get
      \[\begin{aligned}
      A(r)\le C_2t^{-nq}\|\Delta^n_tf\|_p^q\psi (1-r)+C_3Kt^q(1-r)^{-q}A((1+r)/2)
       \end{aligned} \]
       for $0<t<1-r.$ Now let $m>0$ be the smallest integer such that
       $2^{-mq}KC_3\le 1/4$ and take $t=(1-r)/2^m.$ Then we have
        \[\begin{aligned}
        A(r)\le C_4(1-r)^{-nq}\phi(2^{-m}(1-r))\psi(1-r)+(1/4)A((1+r)/2),
        \end{aligned} \]
        where $\phi(t)=\|\Delta^n_tf\|_p^q.$
        Hence, by integration
        \[\begin{aligned}
        \int_0^1A(r)\,dr \le
        C_4\int_0^1t^{-nq}\phi(2^{-m}t)\psi(t)\,dt+(1/2)\int_{1/2}^1A(r)dr
        \end{aligned} \]
         and hence
         \[\begin{aligned}
         (1/2)\int_0^1A(r)\,dr&\le\int_0^1A(r)\,dr-(1/2)\int_{1/2}^1A(r)\,dr\\&
         \le C_4\int_0^1t^{-nq}\phi(2^{-m}t)\psi(t)\,dt\\&
         =C_5\int_0^{2^{-m}}t^{-nq}\phi(t)\psi(2^mt)\,dt\\&
         \le C_5K^m\int_0^1t^{-nq}\phi(t)\psi(t)\,dt.
          \end{aligned} \]
          This concludes the proof.

\begin{cor}
If $\alpha >-1$, then
\begin{equation}
\int_r^1M_p^q(\rho ,\mathcal R^nf)(1-\rho )^{\alpha }d\rho \leq
C\int_0^{1-r}\bigl [t^{-n}||\Delta^n_tf||_p\bigr ]^qt^{\alpha }dt
\end{equation}
$(0<r<1)$, where $C$ is independent of $r$ and $f$.
\end{cor}

\begin{proof}
For a fixed $r$, $0<r<1$, we consider the function
$$
\psi (x)=x^{\alpha },\quad \text{if}\quad 0<x<1-r,\quad \text{and}
\quad \psi (x)=0, \quad \text{if}\quad 1-r<x<1.
$$
Then $\psi $ satisfies $(1.10)$ with $K=2^{\alpha }$ , and $K$ is
independent of $r$. Now $(5.2)$ follows from $(1.11).$
\end{proof}

\begin{cor}
If $\alpha >-1$, then
\begin{equation}
M_p(r,\mathcal R^nf)\leq C\biggl \{(1-r)^{-\alpha
-1}\int_0^{1-r}\bigl [t^{-n}||\Delta^n_tf||_p\bigr ]^qt^{\alpha
}dt\biggr \}^{1/q},
\end{equation}
where $C$ is independent of $r$ and $f$.
\end{cor}

\begin{proof}
By the increasing property of $M_p(\rho ,\mathcal R^n),$
$$
M_p^q(r,\mathcal R^nf)\int_r^1(1-\rho )^{\alpha }d\rho \leq
\int_r^1M_p^q(\rho ,\mathcal R^n)(1-\rho )^{\alpha }d\rho,
$$
which, together with $(5.2)$ gives $(5.3)$.
\end{proof}

Having in mind Lemma 4.1, as a special case we have $(1.12)$.

 Theorem 1.6 is an easy consequence of Lemma 5.5 and Lemma 5.6
that we prove below.

\begin{lem}\label{}
For $f\in H(\bn)$ and $U\in \mathcal L(\bn)$ we have
$$
\begin{aligned} \mathcal R^n(f\circ U)(w)=(\mathcal R^nf)(Uw),
\quad w\in \bn.
 \end{aligned}
 $$
\end{lem}

 \begin{proof}
 It is enough to consider the case $n=1.$
 We have
 $
 \mathcal Rf(w)=f'(w)w,$
 where $f'(w)$ is the derivative of $f$ at $w$ treated as a linear operator on $\cn$.
 Hence, by the chain rule,
 \[\begin{aligned} \mathcal R(f\circ U)(w)=(f\circ U)'(w)w=
 f'(Uw)Uw= (\mathcal Rf)(Uw).  \end{aligned} \]
   \end{proof}

  \begin{cor}\label{}
  If\/ $U\in \mathcal L(\bn)$, then
  \begin{equation}\label{dr}\begin{aligned}
  \mathcal R^k\Delta^m_U f= \Delta ^m_U\mathcal R^k f, \quad k,m\ge 1.
    \end{aligned} \end{equation}
   \end{cor}

\begin{lem}\label{4.3}
If $U\in \mathcal L(\bn)$ and $\|U-I\|<1-r,$ then for $f\in
H(B_N)$ we have
\begin{equation}\label{}\begin{aligned}
M_p(r,\Delta^n_Uf)\le C(1-r)^n M_p(r_n,\mathcal R^nf),
 \end{aligned} \end{equation}
 where
 $$
 \begin{aligned}
 r_n=1-\frac{1-r}{2^n}.
  \end{aligned}
  $$
  \end{lem}

  \begin{proof} (Induction on $n$.) Assume first that $f$ is holomorphic in a neighborhood
  of the closed ball. Let $\|U-I\|<\varepsilon$, where $\varepsilon$ is small enough.
   Then $|Ur\zeta-r\zeta|<\eps$ so, by the Lagrange theorem,
  $$
  \begin{aligned}
  |f(Ur\zeta)-f(r\zeta)|\le \sup_{|w-r\zeta|<\varepsilon,\ |w|<r }|\nabla f(w)|\varepsilon.
   \end{aligned}
   $$
   Choose $\eps=(1-r)/2.$ Since the inequalities $|w-r\zeta|<\eps$ and $|w|<r$ imply
   \[\begin{aligned}
   |w-\zeta|&\le |w-r\zeta|+(1-r)\\&\le (3/2)(1-r)\\&\le (3/2)(1-|w|),
   \end{aligned} \]
   we have that
   \[\begin{aligned}  |f(Ur\zeta)-f(r\zeta)|\le \frac{1-r}2M_*(\nabla f)(\zeta), \end{aligned} \]
   where $M_*$ is the non-tangential maximal function. Now the maximal theorem shows that
   $$
   \begin{aligned}
   M_p(r,\Delta^1_Uf)\le C(1-r)\|\nabla f\|_p.
    \end{aligned}
    $$
    If $f\in H(\bn)$ is arbitrary, then we apply this inequality to the function $f_\rho(z)=f(\rho z),$
    $0<\rho<1,$ to get
    $$
    \begin{aligned}
    M_p(r\rho,\Delta^1_Uf)\le C(1-r)M_p(\rho, \nabla f).
     \end{aligned}
     $$
     Now take $\rho=(t+1)/2$ and $r\rho=t$ for $0<t<1.$ This implies, via Lemma \ref{ahern},
     $$
     \begin{aligned}
     M_p(t,\Delta^1_Uf)&\le C(1-t)M_p((t+1)/2,\nabla f)\\&\le C(1-t)M_p((t+1)/2,\mathcal R f).
       \end{aligned}
    $$
       This proves the lemma for $n=1.$

       Let $n\ge 2.$ Then, by induction hypothesis, relation \eqref{dr}, and the case $n=1$,
       $$
       \begin{aligned}
       M_p(r,\Delta^n_Uf)&\le C(1-r)^{n-1}M_p(r_{n-1}, \mathcal R^{n-1}\Delta^1_Uf)\\&
       =C(1-r)^{n-1}M_p(r_{n-1},\Delta^1_U \mathcal R^{n-1}f)\\&
       \le C(1-r)^{n}M_p((1+r_{n-1})/2, \mathcal R^nf)\\&=
       C(1-r)^{n}M_p(r_n, \mathcal R^nf).
         \end{aligned}
         $$
         This completes the proof.
   \end{proof}

\begin{lem}\label{} If $f\in H^p(\bn)$, $U\in \mathcal L(\bn),$
and $\|U-I\|<1-r,$ $1/4<r<1,$ then
\begin{equation}\label{}\begin{aligned}
\|\Delta_U^n(f-f_r)\|_p\le C\int_r^1(1-s)^{n-1}M_p(s,\mathcal R^nf)\,ds.
 \end{aligned} \end{equation}
 \end{lem}

\begin{proof}
  From the identity
$$
\begin{aligned}
   f(\zeta)-f(r\zeta)=\int_r^1\frac 1s\mathcal Rf(s\zeta)\,ds
 \end{aligned}
 $$
    it follows that
 $$
 \begin{aligned}
      \|\Delta^n_U(f-f_r)\|_p&\le
   4\int_r^1M_p(s,\Delta^n_U\mathcal Rf)\,ds.
 \end{aligned}
 $$
       Hence, by Lemma \ref{4.3},
 $$
 \begin{aligned}
       \|\Delta^n_U(f-f_r)\|_p\le
    C\int_r^1(1-s)^{n}M_p(s_n,\mathcal R^{n+1}f)\,ds,
  \end{aligned}
  $$
        where
 $$
 \begin{aligned}
         s_n=1-\frac{1-s}{2^n}.
 \end{aligned}
 $$
         Now we use the familiar estimate
$$
\begin{aligned}
         M_p(r,\mathcal Rf)\le M_p(r,\nabla f)\le C(1-r)^{-1}M_p((1+r)/2,f)
\end{aligned}
$$
          to get
$$
\begin{aligned}
          M_p(s_n,\mathcal R^{n+1}f)\le C(1-s)^{-1}M_p(s_{n+1},\mathcal R^nf),
\end{aligned}
$$
            which gives
$$
\begin{aligned}
       \|\Delta^n_U(f-f_r)\|_p\le
    C\int_r^1(1-s)^{n-1}M_p(s_{n+1},\mathcal R^{n}f)\,ds,
\end{aligned}
$$
         Now the substitution $s_{n+1}=t$ completes the proof.
\end{proof}

\textit{Proof of Theorem 1.6}

In Corollary 2.5 it is shown that if
$$    \int_0^1
(1-r)^{n-1}M_p(r,\mathcal R^nf)\,dr\ <\infty,
$$
then $f\in H^p.$

By using Lemma 5.5 and Lemma 5.6, (relations $(5.5)$ and $(5.6)$),
we get the inequality :
$$
\begin{aligned}
\omega_n^+(\delta ,f)_p&=\sup_{\|U-I\|<\delta , U\in \mathcal
L}\|\Delta^n_Uf\|_p\\
&\leq \sup_{\|U-I\|<\delta , U\in \mathcal
L}\|\Delta^n_Uf_{1-\delta }\|_p + \sup_{\|U-I\|<\delta , U\in
\mathcal L}\|\Delta^n_U(f-f_{1-\delta })\|_p\\
&\leq C\int_{1-\delta}^1(1-r)^{n-1}M_p(r,\mathcal R^nf)dr\\
\end{aligned}
$$

\section{Remarks}

For $f\in L^1(\sigma )$ and $z\in \bn $ we define the Cauchy
integral
$$
Cf(z)=\int_{\sn }f(\xi )\frac {d\sigma (\xi )}{(1-<z,\xi >)^N}.
$$
In \cite{3} the authors obtained conditions on $f\in L^{1}(\sigma
)$ sufficient to imply that $Cf$ belongs to either the Besov space
$\Lambda^{p,p}_{\alpha }$ or the Hardy-Sobolev space $H^p_{\alpha
}$, where $1<p<\infty $ and $0<\alpha <\infty $. As a corollary of
our results we have  sufficient conditions different from those
given in \cite{3}.

\begin{thm}
Let $1<p<\infty $, $0<\alpha <\infty $ and suppose that $n$ is an
integer such that $0<\alpha <n.$ Then a sufficient condition that
$Cf\in \Lambda^{p,p}_{\alpha }$ is that $f\in L^p(\sigma )$ and
$$
\int_0^1\frac {||\Delta^n_tf||^p_pdt}{t^{1+\alpha p}}<\infty .
$$
A sufficient condition that $f\in H^p_n$ is that $f\in L^p(\sigma
)$ and  $||\Delta^n_tf||_p=\mathcal O(t^n)$.
\end{thm}

\bibliographystyle{amsplain}

\end{document}